\documentclass[12pt]{amsart}
\usepackage{amssymb,amsmath,amsthm,latexsym}
\usepackage{mathrsfs}
\usepackage{mdwlist,enumitem}
\usepackage{a4wide}
\usepackage{graphicx}
\usepackage{mathtools,dsfont}
\usepackage{mathdots}

\newcommand{\R}{\mathbb{R}}
\newcommand{\Q}{\mathbb{Q}}
\newcommand{\N}{\mathbb{N}}

\newcommand{\de}{\delta}
\newcommand{\p}{\varphi}
\newcommand{\e}{\varepsilon}
\newcommand{\ro}{\varrho}

\newcommand{\oo}{\overline}

\newcommand{\Sz}{\mathrm{Sz}}

\newcommand{\inj}{\hat{\otimes}_{\e}}

\newcommand{\abs}[1]{\lvert#1\rvert}
\newcommand{\n}[1]{\|#1\|}

\renewcommand{\span}{\mathrm{span}}
\newcommand{\diam}{\mathrm{diam}}

\newcommand{\ccup}{\scalebox{0.95}{$\bigcup$}}

\renewcommand{\leq}{\leqslant}
\renewcommand{\geq}{\geqslant}
\renewcommand{\frown}{\mathbin{\raisebox{0.3ex}{$\smallfrown$}}}

\newcommand{\vertiii}[1]{{\left\vert\kern-0.25ex\left\vert\kern-0.25ex\left\vert #1 
    \right\vert\kern-0.25ex\right\vert\kern-0.25ex\right\vert}}
\newcommand{\vertiiib}[1]{{\Biggl\vert\kern-0.25ex\Biggl\vert\kern-0.25ex\Biggl\vert #1 
    \Biggr\vert\kern-0.25ex\Biggr\vert\kern-0.25ex\Biggr\vert}}

\newtheorem*{main_theorem}{Main Theorem}
\newtheorem*{corollary*}{Corollary}

\newtheorem{theorem}{Theorem}
\newtheorem{lemma}[theorem]{Lemma}

\newtheorem{proposition}[theorem]{Proposition}

\theoremstyle{definition}
\newtheorem*{definition}{Definition}

\theoremstyle{remark}
\newtheorem*{remark}{Remark}
\newtheorem*{problem}{Question}

\begin{document}
\title[The Szlenk power type and tensor products]{The Szlenk power type and tensor products\\ of Banach spaces}

\author[S. Draga]{Szymon Draga}
\address{Institute of Mathematics, University of Silesia, Bankowa 14, 40-007 Katowice, Poland}
\email{szymon.draga@gmail.com}

\author[T. Kochanek]{Tomasz Kochanek}
\address{Institute of Mathematics, Polish Academy of Sciences, \'Sniadeckich 8, 00-656 Warsaw, Poland\, {\rm and}\, Institute of Mathematics, University of Warsaw, Banacha~2, 02-097 Warsaw, Poland}
\email{tkoch@impan.pl}

\subjclass[2010]{Primary 46B20, 46B28}
\keywords{}
\thanks{}

\begin{abstract}
We prove a formula for the Szlenk power type of the injective tensor product of Banach spaces with Szlenk index at most $\omega$. We also show that the Szlenk power type as well as summability of the Szlenk index are separably determined, and we extend some of our recent results concerning direct sums.
\end{abstract}
\maketitle

\section{Introduction}
The notion of Szlenk index was introduced in \cite{szlenk} in order to show that there is no universal space in the class of all separable reflexive Banach spaces. Since then it has proven to be an~extremely useful tool in the~Banach space theory. The geometry of a~given Banach space with Szlenk index $\omega$ heavily depends on the so-called Szlenk power type which encodes the rate of cutting out the~unit dual ball by iterates of Szlenk derivations; it is strictly connected with the asymptotic moduli of smoothness and convexity ({\it cf. }\cite{lancien} and the references therein). In this paper, motivated mainly by the work of Causey \cite{causey}, we deal with determining the Szlenk power type of injective tensor products of Banach spaces, which in some cases should lead us to getting new information about asymptotic geometry of spaces of compact operators.

For a~Banach space $X$ we denote by $B_X$ and $S_X$ the unit ball and the unit sphere of $X$, respectively. If $K$ is a~weak$^\ast$-compact subset of $X^\ast$ and $\e>0$, then we define the $\e$-{\it Szlenk derivation} of $K$ by
$$
\iota_\e K=\bigl\{x^\ast\in K\colon \mathrm{diam}(K\cap U)>\e\mbox{ for every }w^\ast\mbox{-open neighborhood }U\mbox{ of }x^\ast\bigr\}
$$
and its iterates by $\iota_\e^0K=K$, $\iota_\e^{\alpha+1}K=\iota_\e(\iota_\e^\alpha K)$ for any ordinal $\alpha$, and $\iota_\e^\alpha K=\bigcap_{\beta<\alpha}\iota_\e^\beta K$ for any limit ordinal $\alpha$. Then the $\e$-{\it Szlenk index} of $X$, $\Sz(X,\e)$, is defined as the least ordinal $\alpha$ (if any such exists) for which $\iota_\e^\alpha B_{X^\ast}=\varnothing$; the {\it Szlenk index} of $X$ is defined as $\Sz(X)=\sup_{\e>0}\Sz(X,\e)$. By compactness $\Sz(X,\e)$ is always a successor ordinal and the condition $\Sz(X)\leq\omega$ is equivalent to $\Sz(X,\e)$ being finite for every $\e>0$. The fact that the function $(0,1)\ni\e\xmapsto[]{\phantom{xx}}\Sz(X,\e)$ is submultiplicative ({\it cf. }\cite[Prop.~4]{lancien}) implies that there exists a~finite limit
$$
\mathsf{p}(X)\coloneqq\lim_{\e\to 0+}\frac{\log\Sz(X,\e)}{\abs{\log\e}},
$$
which is called the {\it Szlenk power type} of $X$.

For any Banach spaces $X$ and $Y$ we denote by $X\inj Y$ their injective tensor product (and refer the reader to \cite{ryan} for any unexplained issues concerning this notion). In \cite{causey}, Causey proved that the Szlenk index of $X\inj Y$ behaves generally well---in particular, we have $\Sz(X\inj Y)\leq\omega$ whenever $\Sz(X)\leq\omega$ and $\Sz(Y)\leq\omega$ and both $X$ and $Y$ are separable. This makes sensible the question of determining the value of $\mathsf{p}(X\inj Y)$. Our main result, which concerns not necessarily separable spaces, reads thus as follows.
\begin{main_theorem}\label{thmtensor}
For any Banach spaces $X$, $Y$ with $\Sz(X)\leq\omega$ and $\Sz(Y)\leq\omega$ we have
$$
\mathsf{p}(X\inj Y)=\max\{\mathsf{p}(X),\mathsf{p}(Y)\}.
$$
\end{main_theorem}

The key motivation for studying the Szlenk power type stems from the fact that it corresponds to the~asymptotic moduli of smoothness and convexity introduced by Milman~\cite{milman}. Knaust, Odell and Schlumprecht~\cite{KOS} showed that a~separable Banach space $X$ satisfies $\Sz(X)\leq\omega$ if and only if it can be given an~equivalent asymptotically uniformly smooth norm (with a~power type modulus) or, equivalently, a~norm whose dual norm is weak$^\ast$ asymptotically uniformly convex (with a~power type modulus)---this was later extended to the nonseparable case by Raja~\cite{raja}. Godefroy, Kalton and Lancien~\cite{gkl} gave an~exact quantitative result by showing that the Szlenk power type corresponds to the optimal power types of both the above mentioned moduli; for example, $\mathsf{p}(X)$ is the conjugate of the supremum over all those $q>1$ for which $X$ can be renormed to have asymptotic modulus of smoothness dominated by $Ct^q$ with some $C>0$. 

Consequently, our result gives some information on asymptotic geometry of the space $\mathcal{K}(X,Y)$ of compact operators acting between certain Banach spaces $X$ and $Y$ (recall that if either $X^\ast$ or $Y$ has the approximation property, then $\mathcal{K}(X,Y)$ is isometrically isomorphic to $X^\ast\inj Y$; {\it cf. }\cite[Cor.~4.13]{ryan}).
\begin{corollary*}
If $X$ and $Y$ are Banach spaces so that either $X^\ast$ or $Y$ has the approximation property and both $\Sz(X^\ast)$ and $\Sz(Y)$ are at most $\omega$, then
$$
\mathsf{p}(\mathcal{K}(X,Y))=\max\{\mathsf{p}(X^\ast),\mathsf{p}(Y)\}.
$$
In particular, for all $1<p,q<\infty$ we have
$$
\mathsf{p}(\mathcal{K}(\ell_p,\ell_q))=\max\Bigl\{p,\frac{q}{q-1}\Bigr\}.
$$
\end{corollary*}

In the next section we recall some necessary terminology and tools concerning block/tree estimates and asymptotic structures that are essential to proving our main result. In Section~3 we give a~proof in the separable case, where the crucial step is to obtain a~connection between $\mathsf{p}(X)$ and the optimal exponent for subsequential upper tree estimates (see Proposition~\ref{P2} below). In Section~4 we deal with the nonseparable case. We then conclude the paper with extending some of our recent results \cite{dk} on summability of the~Szlenk index and the~Szlenk power type of direct sums.

%%%%%%%%%%%%%%%%%%%%%%%%%%%%%%%%%%%%%%%%%%%%%%%%%%%%%%%%%%%%%%%%%%%%%%%%%%%%%%%%%%%%%%%%
%%%%%%%%%%%%%%%%%%%%%%%%%%%%%%%%%%%%%%%%%%%%%%%%%%%%%%%%%%%%%%%%%%%%%%%%%%%%%%%%%%%%%%%%
\section{Tools}
Recall that a~sequence $\mathsf{E}=(E_n)$ of finite-dimensional subspaces of $X$ is called a~{\it finite-dimensional decomposition} (FDD for short) if every $x\in X$ has a~unique representation $x=\sum_{n=1}^\infty x_n$ with $x_n\in E_n$ for every $n\in\N$. In such a~case we denote by $P_n^{\mathsf{E}}$ the $n$th canonical projection $X\to E_n$ and for every $z\in c_{00}(\bigoplus_{n=1}^\infty E_n)$ we set $\mathrm{supp}_{\mathsf{E}}z=\{n\in\N\colon P_n^{\mathsf{E}}z\not=0\}$. A~(finite or infinite) sequence $(z_n)$ in $X$ is called a~{\it block sequence} (with respect to $\mathsf{E}$) if for all suitable $n$'s we have
$$
\max\,\mathrm{supp}_{\mathsf{E}}z_n<\min\,\mathrm{supp}_{\mathsf{E}}z_{n+1}.
$$

An~FDD is called {\it shrinking} if $X^\ast$ coincides with the norm closure of $c_{00}(\bigoplus_{n=1}^\infty E_n^\ast)$, that is, the subset of $X^\ast$ that consists of all functionals $(x_n^\ast)_{n=1}^\infty\in\prod_{n=1}^\infty E_n^\ast$ with $x_n^\ast\not=0$ for finitely many $n$'s.

If $\mathsf{E}=(E_n)$ is an~FDD for $X$ and $V$ is a~Banach space with a~normalized, $1$-unconditional basis $(v_n)$, then we say that $\mathsf{E}$ {\it satisfies subsequential} $C$-$V$-{\it upper block estimates}, with some $C\geq 1$, provided that for every normalized block sequence $(z_n)\subset X$ (with respect to $\mathsf{E}$) and any finitely supported sequence of scalars $(a_n)$ we have
$$
\Biggl\|\sum_{n=1}^\infty a_nz_n\Biggr\|\leq C\Biggl\|\sum_{n=1}^\infty a_nv_{m_n}\Biggr\|,\quad\mbox{where }m_n=\min\,\mathrm{supp}_{\mathsf{E}}z_n.
$$

We shall now recall some terminology concerning trees in Banach spaces. First, define
$$
T_l=\bigl\{(n_1,\ldots,n_l)\colon n_1<\ldots<n_l\mbox{ are in }\N\bigr\}\quad\mbox{ for }l\in\N.
$$
We consider the trees $S_l=\ccup_{j=1}^l T_j$ for $l\in\N\cup\{\infty\}$ ordered by the initial segment relation, that is, for $\alpha=(m_1,\ldots,m_k)$ and $\beta=(n_1,\ldots,n_l)$ we write $\alpha\leq\beta$ iff $k\leq l$ and $m_i=n_i$ for each $1\leq i\leq k$. For any $\alpha=(m_1,\ldots,m_k)$ we set $\abs{\alpha}=k$ and call it the {\it length} of $\alpha$. For each $l\in\N\cup\{\infty\}$, we say that $S_l$ is {\it of order} $l$; in other words, the order of $S_l$ is the largest possible length of a~{\it node} in $S_l$. We say that $\beta$ is a~{\it successor} of $\alpha$ if $\abs{\beta}=\abs{\alpha}+1$ and $\alpha\leq\beta$, so $\beta=\alpha\!\frown\! k$ for some $k\in\N$, $k>\max\alpha$, where $\frown$ stands for concatenation.

Let $\sigma$ be any set and let $\sigma^{<\omega}$ be the collection of all finite sequences in $\sigma$. A~family $\mathcal{F}\subset\sigma^{<\omega}$, ordered by the initial segment relation, is called a~{\it tree on} $\sigma$ if it is tree-isomorphic to one of $S_l$'s ($l\in\N\cup\{\infty\}$) and is closed under taking initial segments. The order of $\mathcal{F}$ is, by definition, the same as the order of the corresponding tree $S_l$ and we denote it by $\mathrm{ord}(\mathcal{F})$. It is sometimes convenient to write a~tree on $\sigma$ in the form $(x_\alpha)_{\alpha\in S_l}$, where each $x_\alpha\in\sigma$; this is then identified with
$$
\mathcal{F}=\bigl\{(x_{(m_1)},x_{(m_1,m_2)},\ldots,x_{(m_1,\ldots,m_k)})\colon \alpha=(m_1,\ldots,m_k)\in S_l\bigr\}.
$$
By a {\it branch} of $\mathcal{F}$ we mean any maximal linearly ordered subset of $\mathcal{F}$, which we identify with a~(finite or infinite) set of the form $\{x_{(m_1)},x_{(m_1,m_2)},\ldots\}$.

If $(\beta_i)_{i=1}^\infty$ is the~sequence of all successors of some $\alpha$ with $0\leq\abs{\alpha}<l$, then (under the above convention) the sequence $(x_{\beta_i})_{i=1}^\infty$ is called an~$s$-{\it subsequence} of $\mathcal{F}$. If $\sigma$ is a~subset of a~vector space equipped with some topology $\tau$, then we say that $\mathcal{F}$ is $\tau$-{\it null} provided every $s$-subsequence of $\mathcal{F}$ is $\tau$-null. We shall be mainly concerned with weakly null trees on Banach spaces and weak$^\ast$-null trees on dual Banach spaces.

\begin{definition}[{{\it cf. }\cite{AJO}}]
Let $X$ be a Banach space. We say that a~sequence $(x_j)_{j=1}^n\subset S_X$ is an~$\ell_1^+$-$\ro$-sequence, for some $\ro\in (0,1]$, if
$$
\Biggl\|\sum_{j=1}^na_jx_j\Biggr\|\geq\ro\sum_{j=1}^na_j\quad\mbox{for every }(a_j)_{j=1}^n\subset [0,\infty).   
$$
If $\mathcal{F}$ is a tree on $X$, then we say that it is an~$\ell_1^+$-$\ro$-{\it weakly null tree} provided it is weakly null and its every node is an~$\ell_1^+$-$\ro$-sequence. 
\end{definition}

\noindent
According to results by Alspach, Judd and Odell \cite{AJO}, the behavior of Szlenk derivations of $B_{X^\ast}$ can conveniently be described in terms of the quantities 
$$
\mathrm{I}_{w,\ro}^+(X)\coloneqq\sup\bigl\{\mathrm{ord}(\mathcal{F})\colon \mathcal{F}\mbox{ is an }\ell_1^+\mbox{-}\ro\mbox{-weakly null tree on }S_X\bigr\}\quad (0<\ro<1).
$$
Originally, they considered derivations defined by
$$
P_\e(K)=\Bigl\{x^\ast\in K\colon\scalebox{1.1}{$\exists$}\,{(x_n^\ast)\subset K},\,\,\, x_n^\ast\xrightarrow[]{\,\,w\ast\,\,}x^\ast\mbox{ and }\liminf_n\n{x_n^\ast-x^\ast}\geq\e\Bigr\}.
$$
However, it is easily seen that for every weak$^\ast$-compact set $K\subset X^\ast$ and $\e\in (0,1)$ we have
$$
\iota_\e K\subseteq P_{\e/2}(K)\quad\mbox{ and }\quad P_\e(K)\subseteq\iota_{\e^\prime} K\quad\mbox{for each }0<\e^\prime<\e,
$$
and hence we can rephrase their results in the following form.
\begin{theorem}[{{\it cf. }\cite[Prop. 4.3, 4.10]{AJO}}]\label{AJO_theorem}
If $X$ is a separable Banach space with $\Sz(X)\leq\omega$, then for all $n\in\N$ and $\e,\ro\in (0,1)$ we have:
\begin{itemize}
\item if $\iota_\e^nB_{X^\ast}\not=\varnothing$, then there exists an~$\ell_1^+$-$\frac{1}{16}\e$-weakly null tree on $S_X$ of order $n$;

\vspace*{1mm}
\item if there exists an $\ell_1^+$-$\ro$-weakly null tree on $S_X$ of order $n$, then $\iota_{\de}^n B_{X^\ast}\not=\varnothing$ for every $0<\de<\ro$.
\end{itemize}
\end{theorem}

\begin{remark}
The above assertion holds true for $X^\ast$ being separable and without assuming that $\Sz(X)\leq\omega$, provided that one considers weakly null trees of higher orders being countable ordinals ({\it cf. }\cite[\S 3]{AJO}). Then the $\ell_1^+$-{\it weak index} of $X$ defined by the formula
$$
\mathrm{I}_w^+(X)=\sup_{0<\ro<1}\mathrm{I}_{w,\ro}^+(X)
$$
happens to be exactly equal to $\Sz(X)$ ({\it cf. }\cite[Thm.~4.2]{AJO}). We shall not go into these details here, as we are exclusively concerned with the case where $\Sz(X)\leq\omega$.
\end{remark}

\begin{lemma}[{\it cf. }{\cite[Prop.~3.4]{gkl}}]\label{gkl_lemma}
Let $X$ be a separable Banach space and $\e_1,\ldots,\e_n>0$. In order that $\iota_{\e_1}\ldots\iota_{\e_n}B_{X^\ast}\not=\varnothing$ it is necessary that there exists a~weak$^\ast$-null tree $(x_\alpha^\ast)_{\alpha\in S_n}$ on $X^\ast$ of order $n$ such that $\n{x_\alpha^\ast}\geq\frac{1}{4}\e_{\abs{\alpha}}$ for each $\alpha\in S_n$ and $\n{\sum_{\alpha\in\Gamma}x_\alpha^\ast}\leq 1$ for every branch $\Gamma\subset S_n$, and it is sufficient that there exists a~weak$^\ast$-null tree $(x_\alpha^\ast)_{\alpha\in S_n}$ on $X^\ast$ of order $n$ such that $\n{x_\alpha^\ast}\geq\e_{\abs{\alpha}}$ for each $\alpha\in S_n$ and $\n{\sum_{\alpha\in\Gamma}x_\alpha^\ast}\leq 1$ for every branch $\Gamma\subset S_n$. 
\end{lemma}

Let us now recall some terminology concerning asymptotic structures. For a~separable Banach space $X$ we denote by $\mathsf{cof}(X)$ the family of all finite codimensional subspaces of $X$ and consider the following game between Players I~and II:
$$
\begin{array}{l}
\mbox{Player I chooses }Y_1\in\mathsf{cof}(X)\\
\mbox{Player II chooses }y_1\in S_{Y_1}\\
\mbox{Player I chooses }Y_2\in\mathsf{cof}(X)\\
\mbox{Player II chooses }y_2\in S_{Y_2}\\
\ldots
\end{array}
$$
Given $n\in\N$ and $\mathcal{A}\subseteq S_X^n$ we say that {\it Player~II has a~winning strategy in the }$\mathcal{A}$-{\it game} if he can always end up with $(y_j)_{j=1}^n\in\mathcal{A}$ after $n$ steps, no matter what subspaces $Y_j$'s were picked by Player~I. Let $\mathcal{M}_n$ be the collection of all normalized monotone basic sequences of length $n$; then $(\mathcal{M}_n,\log d_{\mathsf{BM}})$ is a~compact metric space, where $d_{\mathsf{BM}}$ stands for the Banach--Mazur distance (we identify all sequences which are $1$-equivalent).
\begin{definition}[{{\it cf. }\cite{MTJ}}]
Let $X$ be a~Banach space and $n\in\N$. We say that a~sequence $(e_j)_{j=1}^n\in\mathcal{M}_n$ is an~{\it element of the} $n^{th}$ {\it asymptotic structure of} $X$, and then we write $(e_j)_{j=1}^n\in\{X\}_{n}$, provided that
\begin{equation*}
\begin{split}
\scalebox{1.1}{$\forall$}\,\e>0\,\,&\scalebox{1.1}{$\forall$}\, Y_1\in\mathsf{cof}(X)\,\,\scalebox{1.1}{$\exists$}\, y_1\in S_{Y_1}\,\ldots\,\scalebox{1.1}{$\forall$}\,Y_n\in\mathsf{cof}(X)\,\,\scalebox{1.1}{$\exists$}\,y_n\in S_{Y_n}\\
&d_{\mathsf{BM}}((y_j)_{j=1}^n, (e_j)_{j=1}^n)<1+\e.
\end{split}
\end{equation*}
In other words, $(e_j)_{j=1}^n\in\{X\}_{n}$ if and only if for every $\de>0$ Player~II has a~winning strategy in the $\mathcal{A}_\de$-game, where $\mathcal{A}_\de$ is the~ball in $\mathcal{M}_n$ with center $(e_j)_{j=1}^n$ and radius $\de$.
\end{definition}

In the case where $X^\ast$ is separable this property can be restated in terms of trees ({\it cf. }\cite[Cor.~5.2]{OS_trees}). Namely, $\{X\}_{n}$ is the minimal closed subset of $\mathcal{M}_n$ such that for any $\e>0$ every weakly null tree on $S_X$ of order $n$ has a~node $(y_j)_{j=1}^n$ with $d_{\mathsf{BM}}((y_j)_{j=1}^n,\{X\}_n)<1+\e$. Therefore, Theorem~\ref{AJO_theorem} and a~simple prunning argument guarantee that for every separable Banach space $X$ with $\Sz(X)\leq\omega$ and any $\ro\in (0,1)$ there is some uniform bound on the lengths of $\ell_1^+$-$\ro$-sequences lying in an~asymptotic structure of $X$.

We say that $X$ satisfies {\it subsequential} $\ell_q$-{\it upper tree estimates} if there exists a~constant $C>0$ so that every weakly null tree on $S_X$ contains a~branch $(x_n)$ which for every finitely supported sequence of scalars $(a_n)$ satisfies $\n{\sum_n a_nx_n}\leq C(\sum_{n}\abs{a_n}^q)^{1/q}$. The following assertion is a~part of a~theorem due to Odell and Schlumprecht.
\begin{theorem}[{{\it cf. }\cite[Thm. 3]{OS}}]\label{OS_theorem}
Let $X$ be a Banach space with $X^\ast$ separable. Then the following assertions are equivalent:
\begin{itemize*}
\item[{\rm (i)}] $\Sz(X)\leq\omega$.
\item[{\rm (ii)}] There exist $q>1$ and $K<\infty$ so that for all $n\in\N$, $(e_i)_{i=1}^n\in\bigl\{X\bigr\}_{\! n}$ and $(a_i)_{i=1}^n\subset\R$ we have
$$
\Biggl\|\sum_{i=1}^na_ie_i\Biggr\|\leq K\Biggl(\sum_{i=1}^n\abs{a_i}^q\Biggr)^{\!\! 1/q}.
$$
\item[{\rm (iii)}] There exists $\overline{q}>1$ so that $X$ satisfies subsequential $\ell_{\overline{q}}$-upper tree estimates. In fact, one can take any $\overline{q}\in (1,q)$ with $q$ satisfying assertion {\rm (ii)} above.
\end{itemize*}
\end{theorem}

%%%%%%%%%%%%%%%%%%%%%%%%%%%%%%%%%%%%%%%%%%%%%%%%%%%%%%%%%%%%%%%%%%%%%%%%%%%%%%%%%%%%%%%%
%%%%%%%%%%%%%%%%%%%%%%%%%%%%%%%%%%%%%%%%%%%%%%%%%%%%%%%%%%%%%%%%%%%%%%%%%%%%%%%%%%%%%%%%
\section{The separable case}

The following result is a~`power type' analogue to \cite[Cor.~4.5]{causey}.
\begin{proposition}\label{Prop1}
Let $V$ be a Banach space with a~normalized $1$-unconditional basis $(v_n)$ and assume that $\Sz(V)\leq\omega$. If $X$ is a~Banach space with a~shrinking FDD satisfying subsequential $V$-upper block estimates with respect to $(v_n)$, then $\mathsf{p}(X)\leq\mathsf{p}(V)$. 
\end{proposition}
\begin{proof}
Suppose $\e\in(0,1)$ and $n\in\N$ are such that  $\iota_\e^nB_{X^\ast}\not=\varnothing$. By Theorem~\ref{AJO_theorem}, there exists an~$\ell_1^+$-$\frac{1}{16}\e$-weakly null tree $\mathcal{F}=(x_\alpha)_{\alpha\in S_n}$ on $S_X$ of order $n$. By slightly decreasing the value of $\frac{1}{16}\e$ and using an~easy prunning procedure we may assume that every $s$-subsequence of $\mathcal{F}$ and every branch in $\mathcal{F}$ forms a~block sequence with respect to the given FDD $\mathsf{E}$ of $X$.

Now, we define a~new tree $\mathcal{V}=(w_\alpha)_{\alpha\in S_n}$ on $S_V$ by setting
$$
w_\alpha=v_{N(\alpha)},\quad\mbox{where }\, N(\alpha)\coloneqq\min\,\mathrm{supp}_{\mathsf{E}}x_\alpha\quad (\alpha\in S_n).
$$
If $C\geq 1$ is so that $\mathsf{E}$ satisfies subsequential $C$-$V$-upper block estimates, then every node of $\mathcal{V}$ is an~$\ell_1^+$-$C^{-1}\!\ro$-sequence in $V$, where $\ro$ can be any prescribed positive number smaller than $\frac{1}{16}\e$. For each $\alpha$ with $0\leq \abs{\alpha}<n$ we obviously have $N(\alpha\!\frown\! k)\to\infty$ as $k\to\infty$ and since the basis $(v_n)$ is shrinking, we infer that every $s$-subsequence of $\mathcal{V}$ is weakly null. Therefore, $\mathcal{V}$ is an~$\ell_1^+$-$C^{-1}\!\ro$-weakly null tree in $S_V$ and by appealing to Theorem~\ref{AJO_theorem} once again, we obtain $\iota_\de^n B_{V^\ast}\not=\varnothing$ for every $0<\de<C^{-1}\ro$. Hence, 
$\Sz(X,\e)\leq\Sz(V,\de)$ for every $0<\de<\e/(16C)$, which completes the proof.
\end{proof}

For any $p\in [1,\infty)$ we denote by $p^\prime$ the conjugate exponent, {\it i.e.} $p^\prime=p/(p-1)$ if $p>1$ and $p^\prime=\infty$ if $p=1$. Our next goal is to show that condition (ii) of Theorem~\ref{OS_theorem} holds true with any $q<\mathsf{p}(X)^\prime$. To this end we need to derive a~slightly more delicate quantitative version of Johnson's result \cite[Lemma~III.1]{johnson} which originally says what follows: For any unconditionally monotone basic sequence $(e_i)$ and every $n\in\N$ there exists $p>1$ (namely, any $p$ with $2n^{1/p}<3$) such that if $(e_i)$ does not admit any normalized block subsequence $10$-equivalent to the unit vector basis of $\ell_1^n$, then it satisfies subsequential $3$-$\ell_q$-upper block estimates with $q=p^\prime$. 

To formalize our result it is convenient to introduce the following terminology. Let $a$ be a parameter running through some set $A$ and let $\Phi_a\colon\N\times (c_a,\infty)\to (0,\infty)$, where $c_a\geq 1$, and $\ro_a\colon (c_a,\infty)\to(0,1]$ ($a\in A$). We say that $\{(\Phi_a,\ro_a)\}_{a\in A}$ is an $\ell_1^+$-{\it method} provided that for all $a\in A$, $n\in\N$, $C>c_a$, any Banach space $X$ and every normalized monotone basic sequence $(e_i)$ in $X$ the following condition is satisfied:
\begin{itemize}[leftmargin=\parindent+6pt]
\item[(J)] if $(e_i)$ does not admit any block $\ell_1^+$-$\ro_a(C)$-sequence of length $n$, then for every exponent $p>1$ so that $p\geq\Phi_a(n,C)$ it satisfies subsequential $C$-$\ell_q$-upper block estimates with $q=p^\prime$.
\end{itemize}

\begin{lemma}\label{method1}
Suppose there exists an $\ell_1^+$-method $\{(\Phi_a,\ro_a)\}_{a\in A}$ such that:
\begin{itemize*}
\item[{\rm (a)}] for all $\eta>0$, $p\geq 1$ there exist $\de>0$, $a\in A$ such that $\Phi_a(\ro_a(C)^{-p-\de},C)\leq p+\eta$ whenever $C$ is sufficiently large;
\item[{\rm (b)}] $\lim_{C\to\infty}\ro_a(C)=0$ for each $a\in A$.
\end{itemize*}
Then for every separable Banach space $X$ with $\Sz(X)\leq\omega$ and every $q<\mathsf{p}(X)^\prime$ there exists $K_q<\infty$ so that for all $n\in\N$, $(e_i)_{i=1}^n\in\bigl\{X\bigr\}_{\! n}$ and $(a_i)_{i=1}^n\subset\R$ we have
$$
\Biggl\|\sum_{i=1}^na_ie_i\Biggr\|\leq K_q\Biggl(\sum_{i=1}^n\abs{a_i}^q\Biggr)^{\!\! 1/q}.
$$
\end{lemma}
\begin{proof}
By the remarks following the definition of asymptotic structures, there exists a~function $N\colon (0,1)\to\N$ so that if $(e_i)_{i=1}^n\in\{X\}_{n}$ is an~$\ell_1^+$-$\ro$-sequence, then $n<N(\ro)$. Moreover, Theorem~\ref{AJO_theorem} yields that $N(\ro)\leq \Sz(X,\e)$ for every $0<\e<\ro$. Fix any $q<\mathsf{p}(X)^\prime$ and put $\eta=q^\prime-\mathsf{p}(X)>0$. In view of condition (a), there exist $\de>0$ and $a\in A$ so that for sufficiently large $C$'s we have
\begin{equation}\label{Clarge}
\Phi_a(\ro_a(C)^{-\mathsf{p}(X)-\de},C)\leq \mathsf{p}(X)+\eta=q^\prime.
\end{equation}
Plainly, we have $\Sz(X,\e)\leq\e^{-\mathsf{p}(X)-\de}$ if $\e$ is sufficiently small (just by the definition of $\mathsf{p}(X)$). Hence, with the aid of condition (b) and taking $C$ sufficiently large, we may guarantee that the last inequality holds true for every $\e\leq\ro_a(C)$; we may also assume that \eqref{Clarge} is valid for our choice of $C$. Consequently, 
$$
N(\ro_a(C))\leq\ro_a(C)^{-\mathsf{p}(X)-\de}
$$
which means that there are no $\ell_1^+$-$\ro_a(C)$-sequences of length $\ro_a(C)^{-\mathsf{p}(X)-\de}$ in the corresponding asymptotic structure of $X$. Therefore, condition (J) implies that all members of any asymptotic structure of $X$ satisfy subsequential $C$-$\ell_q$-upper block estimates.
\end{proof}

In order to show that a~suitable $\ell_1^+$-method exists, we shall need an~`$\ell_1^+$-version' of the well-known James' blocking argument used in the proof of his $\ell_1$-distortion theorem (see, {\it e.g.}, \cite[Prop.~2]{OS_handbook}). The original argument applies {\it mutatis mutandis} to our situation, so we omit the proof.
\begin{lemma}\label{james_trick}
Let $N,k\in\N$, $\ro>0$ and suppose that $(x_i)_{i=1}^{N^k}$ is a~normalized $\ell_1^+$-$\ro$-sequence in some Banach space. Then $(x_i)_{i=1}^{N^k}$ admits a~normalized block subsequence of length $N$ which forms an~$\ell_1^+$-$\ro^{1/k}$-sequence.
\end{lemma}

\begin{lemma}\label{method2}
There exists an $\ell_1^+$-method $\{(\Phi_a,\ro_a)\}_{a>1}$ satisfying conditions {\rm (a)} and {\rm (b)}.
\end{lemma}
\begin{proof}
Fix $a>1$ and pick any sequence $(\omega_i)_{i=0}^\infty$ of natural numbers with $1<\omega_i/\omega_{i-1}<a$ for $i\in\N$. Pick $\alpha\in (0,1)$ and define
$$
\Phi_a(n,C)=\frac{\omega_0\log n}{\log (C-\sigma)}\quad\mbox{for }C-1>\sigma:=\sum_{i=0}^\infty\alpha^{\omega_i}.
$$
Define also $\ro_a(C)=\alpha(C-\sigma)^{-a/\omega_0}$.

Fix $n\in\N$, $C>\sigma+1$ and consider any exponent $p>1$ with $p\geq\Phi_a(n,C)$. Let $(e_i)$ be a normalized monotone basic sequence and assume that there is a~(finite) block subsequence $(y_j)$ of $(e_i)$ such that
$$
\Bigl\|\sum_j y_j\Bigr\|>C\Bigl(\sum_j\n{y_j}^q\Bigr)^{\!\! 1/q},\quad\mbox{where }q:=p^\prime.
$$
We shall produce a block $\ell_1^+$-$\ro_a(C)$-subsequence of $(y_j)$ of length $n$. Hence, it is enough to find a~norm one functional $f$ so that $f(z_j)\geq\ro_a(C)$ for each $1\leq j\leq n$ and some $(z_j)_{j=1}^n$ being a~normalized block subsequence of $(y_j)$. (By the geometric Hahn--Banach theorem, it is actually equivalent to the existence of the~sequence $(y_j)$.)

Take a norm one functional $f$ so that $f(\sum_j y_j)=\n{\sum_j y_j}$. Set $\gamma=\alpha n^{-a/p}$ and define
$$
E_0=\bigl\{j\colon \gamma^{\omega_0}\n{y_j}<f(y_j)\leq\n{y_j}\bigr\}\,\mbox{ and }\,
E_i=\bigl\{j\colon \gamma^{\omega_i}\n{y_j}<f(y_j)\leq \gamma^{\omega_{i-1}}\n{y_j}\bigr\}\,\,\,\mbox{for }i\geq 1.
$$
We {\it claim} that the cardinality $\abs{E_i}\geq n^{\omega_i}$ for at least one $i\geq 0$. If this is not true, then by applying H\"older's inequality we obtain
\begin{equation*}
\begin{split}
\Bigl\|\sum_j y_j\Bigr\| &=f\Bigl(\sum_j y_j\Bigr)\leq \sum_{i=0}^\infty\sum_{j\in E_i}f(y_j)\\
& \leq \sum_{j\in E_0}\n{y_j}+\sum_{i=1}^\infty\gamma^{\omega_{i-1}}\sum_{j\in E_i}\n{y_j}\\
& \leq \abs{E_0}^{1/p}\Bigl(\sum_{j\in E_0}\n{y_j}^q\Bigr)^{\!\! 1/q}+\sum_{i=1}^\infty \gamma^{\omega_{i-1}}\abs{E_i}^{1/p}\Bigl(\sum_{j\in E_i}\n{y_j}^q\Bigr)^{\!\! 1/q}\\
&<\Bigl(n^{\omega_0/p}+\sum_{i=1}^\infty\alpha^{\omega_{i-1}}\!\cdot\! n^{(\omega_i-a\omega_{i-1})/p}\Bigr)\Bigl(\sum_j\n{y_j}^q\Bigr)^{\!\! 1/q}\\
&<\bigl(n^{\omega_0/p}+\sigma\bigr)\Bigl(\sum_j\n{y_j}^q\Bigr)^{\!\! 1/q}\leq C\Bigl(\sum_j\n{y_j}^q\Bigr)^{\!\! 1/q}
\end{split}
\end{equation*}
because $p\geq\Phi_a(n,C)$, and hence we arrive at a~contradiction.

Pick an index $i$ with $\abs{E_i}\geq n^{\omega_i}$. By normalizing the vectors from $\{y_j\colon j\in E_i\}$ we obtain a~sequence $(z_j)$ of length at least $n^{\omega_i}$ which consists of unit block vectors and satisfies $f(z_j)\geq\gamma^{\omega_i}$ for each $j$. This means that $(z_j)$ forms an~$\ell_1^+$-$\gamma^{\omega_i}$-sequence and an~appeal to Lemma~\ref{james_trick} produces an~$\ell_1^+$-$\gamma$-sequence of length $n$. Notice that since $n^{\omega_0/p}\leq C-\sigma$, we have
$$
\gamma=\alpha n^{-a/p}\geq \alpha (C-\sigma)^{-a/\omega_0}=\ro_a(C),
$$
so the resulting sequence is in fact an~$\ell_1^+$-$\ro_a(C)$-sequence. This shows that $\{(\Phi_a,\ro_a)\}_{a>1}$ yields an~$\ell_1^+$-method.

It remains to verify conditions (a) and (b). For arbitrarily fixed $\eta>0$ and $p\geq 1$ note that
$$
\Phi_a(\ro_a(C)^{-p-\de},C)\leq p+\eta\quad\mbox{ if and only if }\quad\displaystyle{(C-\sigma)^{(a-\frac{p+\eta}{p+\de})/\omega_0}}\leq\alpha.
$$
This can be easily guaranteed once we take $\de<\eta$, $a<\frac{p+\eta}{p+\de}$ and $C$ sufficiently large. Hence, condition (a) holds true. Condition (b) is obvious by the very definition.
\end{proof}

\begin{proposition}\label{P2}
If $X$ is a separable Banach space with $\Sz(X)\leq\omega$, then it satisfies subsequential $\ell_q$-upper tree estimates for every $q<\mathsf{p}(X)^\prime$.
\end{proposition}
\begin{proof}
Just combine Lemmas~\ref{method1}, \ref{method2} and Theorem~\ref{OS_theorem}.
\end{proof}

Finally, we are prepared to complete the proof of our main result in the separable case by repeating the general scheme of the proof of Causey's theorem mentioned in the introduction.
\begin{proof}[Proof of Main Theorem (the separable case)]
Set $p=\max\{\mathsf{p}(X),\mathsf{p}(Y)\}$. Then Proposition~\ref{P2} says that for every $q<p^\prime$ both $X$ and $Y$ satisfy subsequential $\ell_q$-upper tree estimates. By the Freeman--Odell--Schlumprecht--Zs\'ak theorem \cite[Thm.~1.1]{FOSZ}, both these spaces can be embedded in some Banach spaces, say $W$ and $Z$, having shrinking bimonotone FDD's which satisfy subsequential $C$-$\ell_q$-upper block estimates, with some $C\geq 1$. In view of Causey's result \cite[Lemma~6.6]{causey}, the space $W\inj Z$ has an~FDD which satisfies subsequential $2C$-$\ell_q$-upper block estimates and hence Proposition~\ref{Prop1} implies that
$$
\mathsf{p}(W\inj Z)\leq\mathsf{p}(\ell_q)=q^\prime.
$$
Since $q$ can be taken arbitrarily close to $p^\prime$ and $X\inj Y\hookrightarrow W\inj Z$, the result follows.
\end{proof}

%%%%%%%%%%%%%%%%%%%%%%%%%%%%%%%%%%%%%%%%%%%%%%%%%%%%%%%%%%%%%%%%%%%%%%%%%%%%%%%%%%%%%%%%
%%%%%%%%%%%%%%%%%%%%%%%%%%%%%%%%%%%%%%%%%%%%%%%%%%%%%%%%%%%%%%%%%%%%%%%%%%%%%%%%%%%%%%%%
\section{The nonseparable case}

We start with a simple fact which guarantees that summability of the Szlenk index and its power type are inherited by subspaces in a `uniform' way.

\begin{lemma}\label{lemsubspace}
Let $X$ be a Banach space, $Y$ be a~subspace of $X$ and $\e_1,\ldots,\e_n>0$. Then we have
\[ \iota_{\e_1/2}\ldots\iota_{\e_n/2} B_{X^\ast}\neq\varnothing\quad\mbox{ whenever }\quad\iota_{\e_1}\ldots\iota_{\e_n} B_{Y^\ast}\neq\varnothing. \]
\end{lemma}
\begin{proof}
In the proof of \cite[Lemma 2.39]{hsmvz} it was shown that if $K\subseteq B_{X^\ast}$ and $L\subseteq B_{Y^\ast}$ are weak$^\ast$-compact sets such that $L\subseteq j^\ast(K)$, then $\iota_{\e}L\subseteq j^\ast(\iota_{\e/2}K)$ for every $\e>0$, where $j^\ast$ stands for the~adjoint of the inclusion operator $j\colon Y\to X$. Putting $K=B_{X^\ast}$ and $L=B_{Y^\ast}$ we easily obtain
\[ \iota_{\e_1}\ldots\iota_{\e_n} B_{Y^\ast}\subseteq j^{\ast}(\iota_{\e_1/2}\ldots\iota_{\e_n/2} B_{X^\ast}), \]
which gives the assertion.
\end{proof}

The following lemma is, in a sense, a quantitative version of \cite[Lemma 3.4]{lancien1}.
\begin{lemma}\label{lemsep}
Let $X$ be a Banach space and $\e_1,\ldots,\e_n>0$. Then there exists a separable space $Y\subseteq X$ such that
\[ \iota_{\e_1/4}\ldots\iota_{\e_n/4} B_{Y^\ast}\neq\varnothing\quad\mbox{ whenever }\quad\iota_{\e_1}\ldots\iota_{\e_n} B_{X^\ast}\neq\varnothing. \]
\end{lemma}

\begin{proof}
Given $n\in\N$, let $(\alpha_m)_{m=1}^\infty$ be an enumeration of all elements of the tree $S_n$ so that $\max \alpha_k\leq\max\alpha_l$ whenever $k<l$. Define $\p\colon\N\to S_n$ by $\p(m)=\alpha_m$; note that $\p$ is surjective and $\max\p(m)\leq m$ for each $m\in\N$.

Assume $\iota_{\e_1}\ldots\iota_{\e_n} B_{X^\ast}\neq\varnothing$. By induction on $\max\alpha$, we shall construct a~tree $(x^\ast_\alpha)_{\alpha\in S_n}$ on $X^\ast$ of order $n$, and a~family $\{x_\alpha\colon\alpha\in S_n\}\subset B_X$, such that the following conditions are satisfied:
\begin{enumerate}
\item[\rm(i)] $x_\alpha^\ast(x_\alpha)>\frac14\e_{\abs{\alpha}}$ for each $\alpha\in S_n$;

\vspace*{2mm}
\item[\rm(ii)]
$\arraycolsep=1.4pt
\left\{\begin{array}{ll}
\sum_{\beta\leq\alpha} x_\beta^\ast \in \iota_{\e_{\abs{\alpha}+1}/2} \ldots \iota_{\e_n/2}B_{X^\ast} & \mbox{ for }\abs{\alpha}\leq n-1,\\[2mm]
\sum_{\beta\leq\alpha} x_\beta^\ast \in B_{X^\ast} & \mbox{ for }\abs{\alpha}=n;
\end{array}
\right.$

\vspace*{2mm}
\item[\rm(iii)] $\abs{x_{\alpha\smallfrown m}^\ast(x_{\p(k)})}\leq 2^{-m}$ for  $\alpha\!\frown\! m\in S_n$ and $1\leq k<m$.
\end{enumerate}
Since $\iota_{\e_1}\ldots\iota_{\e_n} B_{X^\ast}\neq\varnothing$, we have
$$
0\in\frac12\iota_{\e_1}\ldots\iota_{\e_n} B_{X^\ast}+\frac12B_{X^\ast}\subseteq\iota_{\e_1/2}\ldots\iota_{\e_n/2} B_{X^\ast}.
$$
Given $m\in\N\cup\{0\}$, suppose that all the elements of the form $x_{\alpha\smallfrown k}$ and $x_{\alpha\smallfrown k}^\ast$, for $1\leq k\leq m$ and $\alpha\!\frown\! k\in S_n$, have been constructed in such a~way that they satisfy conditions (i)--(iii). Take $\alpha\in S_n\cup\{\varnothing\}$ with $\abs{\alpha}<n$ and $\max\alpha<m+1$ (we set $\max\varnothing=0$); we are to define $x_{\alpha\smallfrown m+1}$ and $x_{\alpha\smallfrown m+1}^\ast$. Observe that condition (ii) implies that
$$
\diam \Big( V\cap \iota_{\e_{\abs{\alpha}+2}/2} \ldots \iota_{\e_n/2}B_{X^\ast}\Big) > \frac{\e_{\abs{\alpha}+1}}{2}
$$
for each weak$^\ast$-neighborhood $V$ of $\sum_{\beta\leq\alpha} x_\beta^\ast$. In particular, there is $x^\ast\in \iota_{\e_{\abs{\alpha}+2}/2} \ldots \iota_{\e_n/2}B_{X^\ast}$ such that
$$
\Bigg\| x^\ast-\sum_{\beta\leq \alpha} x_\beta^\ast \Bigg\| > \frac{\e_{\abs{\alpha}+1}}{4}\,\,\,\mbox{ and }\,\,\,\Bigg|\Big( x^\ast-\sum_{\beta\leq\alpha} x_\beta^\ast\Big)(x_{\p(k)}) \Bigg|\leq\frac{1}{2^{m+1}}\quad\mbox{for } 1\leq k\leq m.
$$
Define $x_{\alpha\smallfrown m+1}^\ast=x^\ast-\sum_{\beta\leq\alpha} x_\beta^\ast$. Plainly, conditions (ii) and (iii) are satisfied. To finish the construction pick any $x_{\alpha\smallfrown m+1}\in B_X$ such that $x_{\alpha\smallfrown m+1}^\ast(x_{\alpha\smallfrown m+1})>\frac14\e_{\abs{\alpha}+1}$.

Set $Y=\overline{\span}\,\{x_\alpha\colon \alpha\in S_n\}$ and $y_\alpha^\ast=x_\alpha^\ast\!\upharpoonright\! Y$ for every $\alpha\in S_n$. Condition (iii) guarantees that $(y_\alpha^\ast)_{\alpha\in S_n}$ is a~weak$^\ast$-null tree on $Y^\ast$ of order $n$. Moreover, by conditions (i) and (ii), we have
\begin{enumerate}
\item[\rm(i')] $\n{y_\alpha^\ast}>\frac14\e_{\abs{\alpha}}$ for each $\alpha\in S_n$;

\vspace*{2mm}
\item[\rm(ii')]
$\n{\sum_{\alpha\in\Gamma}y_\alpha^\ast}\leq1$ for every branch $\Gamma\subset S_n$.
\end{enumerate}
Therefore, Lemma~\ref{gkl_lemma} yields $\iota_{\e_1/4}\ldots\iota_{\e_n/4} B_{Y^\ast}\neq\varnothing$.
\end{proof}

We are ready to show that summability of the Szlenk index and the Szlenk power type are separably determined. First, recall that a~Banach space $X$ is said to have {\it summable Szlenk index} if there is a~constant $M$ such that for all positive $\e_1,\ldots,\e_n$ we have $\sum_{i=1}^n\e_i\leq M$ whenever $\iota_{\e_1}\ldots\iota_{\e_n}B_{X^\ast}\not=\varnothing$. Then we also say that $X$ has summable Szlenk index {\it with constant $M$}. Given any family of Banach spaces, we shall say that they have {\it uniformly summable Szlenk index} provided that all of them have summable Szlenk index with the same constant.

\begin{proposition}
A Banach space has summable Szlenk index if every its separable subspace does.
\end{proposition}
\begin{proof}
Given a Banach space $X$ and positive numbers $\e_1,\ldots,\e_n$, let $Y(\e_1,\ldots,\e_n)$ be a~separable subspace of $X$ constructed according to Lemma~\ref{lemsep}. Denote by $\mathcal{E}$ the collection of all finite sequences of positive rational numbers and set
$$
Y=\overline{\span}\, \!\bigcup_{(\de_1,\ldots,\de_n)\in\mathcal{E}}\!Y(\de_1,\ldots,\de_n).
$$
As $Y$ is a~separable subspace of $X$, it is enough to show that $Y$ has nonsummable Szlenk index provided that $X$ does too.

Suppose, towards a contradiction, that $Y$ has summable Szlenk index with constant $M$ and consider any $\e_1,\ldots,\e_n>0$ with $\iota_{\e_1}\ldots\iota_{\e_n} B_{X^\ast}\neq\varnothing$. For each $k$ pick a~rational number ${\de_k\in(\e_k/2,\e_k)}$; of course, $\iota_{\de_1}\ldots\iota_{\de_n} B_{X^\ast} \neq\varnothing$. Hence, $\iota_{\de_1/4} \ldots\iota_{\de_n/4} B_{Y(\de_1,\ldots,\de_n)^\ast} \neq\varnothing$. Therefore, by Lemma~\ref{lemsubspace} we have $\iota_{\de_1/8}\ldots\iota_{\de_n/8} B_{Y^\ast}\neq\varnothing$, which implies $\de_1+\ldots+\de_n\leq 8M$. Thus $\e_1+\ldots+\e_n\leq 16M$ which proves that $X$ has summable Szlenk index.
\end{proof}

\begin{proposition}\label{proppower}
If $X$ is a Banach space with $\Sz(X)=\omega$, then there is a separable space $Y\subseteq X$ with $\mathsf{p}(Y)=\mathsf{p}(X)$.
\end{proposition}
\begin{proof}
For any $\e>0$, $n\in\N$ let $Y(n,\e)$ be a~separable subspace of $X$ constructed according to Lemma~\ref{lemsep} applied to $\e_1=\ldots=\e_n=\e$. Define $$
Y=\oo{\span}\,\bigcup_{n=1}^\infty\bigcup_{\de\in\Q_{+}}Y(n,\de).
$$
Clearly, $Y$ is a separable subspace of $X$. We shall show that $\mathsf{p}(Y)=\mathsf{p}(X)$.

Consider any $\e>0$ and $n\in\N$ with  $\iota_{\e}^nB_{X^\ast}\neq\varnothing$. Pick a~rational number $\de\in(\e/2,\e)$; of course, we have $\iota_{\de}^n B_{X^\ast}\neq\varnothing$ and hence $\iota_{\de/4}^nB_{Y(n,\de)^\ast}\neq\varnothing$. By Lemma~\ref{lemsubspace}, we have $\iota_{\de/8}^n B_{Y^\ast}\neq\varnothing$, whence $\iota_{\e/16}^n B_{Y^\ast}\neq\varnothing$. Consequently, we have shown that $\Sz(Y,\e/16)\geq\Sz(X,\e)$ which yields $\mathsf{p}(Y)\geq\mathsf{p}(X)$.
\end{proof}

\begin{proof}[Proof of Main Theorem (continued)]
Clearly, $\mathsf{p}(X\hat{\otimes}_{\e}Y) \geq \max\{\mathsf{p}(X),\mathsf{p}(Y)\}$. Therefore, in view of Proposition \ref{proppower}, it suffices to show that for any separable subspace $Z$ of $X\hat{\otimes}_{\e}Y$ we have $\mathsf{p}(Z) \leq \max\{\mathsf{p}(X),\mathsf{p}(Y)\}$. Let $\{\mathbf{z}_k\colon k\in\N\}$ be a~dense subset of $Z$. Since $X\otimes Y$ is dense in $X\hat{\otimes}_{\e}Y$, we can write $\mathbf{z}_k = \lim_{n\to\infty}\sum_{m=1}^{M(k,n)} x_{m,n}^{(k)} \otimes y_{m,n}^{(k)}$ for $k\in\N$. Set
$$
X_0=\oo{\span}\,\{x_{m,n}^{(k)}\colon k,n\in\N,\ 1\leq m\leq M(k,n)\}
$$
and
$$
Y_0=\oo{\span}\,\{y_{m,n}^{(k)}\colon k,n\in\N,\ 1\leq m\leq M(k,n)\}.
$$
Clearly, $X_0$ and $Y_0$ are separable subspaces of $X$ and $Y$, respectively. Therefore, by the `separable part', we have $\mathsf{p}(X_0 \hat{\otimes}_{\e} Y_0) = \max\{\mathsf{p}(X_0),\mathsf{p}(Y_0)\} \leq \max\{\mathsf{p}(X),\mathsf{p}(Y)\}$ and since $Z$ embeds in $X_0 \hat{\otimes}_{\e} Y_0$, the assertion follows.
\end{proof}

%%%%%%%%%%%%%%%%%%%%%%%%%%%%%%%%%%%%%%%%%%%%%%%%%%%%%%%%%%%%%%%%%%%%%%%%%%%%%%%%%%%%%%%%
%%%%%%%%%%%%%%%%%%%%%%%%%%%%%%%%%%%%%%%%%%%%%%%%%%%%%%%%%%%%%%%%%%%%%%%%%%%%%%%%%%%%%%%%
\section{Direct sums}
In this section, we extend some of our theorems recently obtained in \cite{dk}. First, observe that using results of Section~4 one can easily prove nonseparable analogues to \cite[Thm.~3.2 and~5.9]{dk}.
\begin{theorem}
For any sequence $(X_n)_{n=1}^\infty$ of Banach spaces with uniformly summable Szlenk index the space $X=(\bigoplus_{n=1}^\infty X_n)_{c_0}$ also has summable Szlenk index.
\end{theorem}
\begin{proof}
We shall show that any separable subspace $Y$ of $X$ has summable Szlenk index. Let $\{\mathbf{y}_k\colon k\in\N\}$ be a~dense subset of $Y$ and write $\mathbf{y}_k=(y_{n,k})_{n=1}^\infty$. For each $n\in\N$ set $Y_n=\oo{\span}\{y_{n,k}\colon k\in\N \}$ which is plainly a~separable subspace of $X_n$. By \cite[Thm.~3.2]{dk}, the space $Z=(\bigoplus_{n=1}^\infty Y_n)_{c_0}$ has summable Szlenk index. Now, it suffices to observe that $Y$ embeds in $Z$.
\end{proof}

In a very similar way one can derive the next result which is a~nonseparable version of \cite[Thm.~5.9]{dk} (for any unexplained notion we refer the reader to \cite[\S 5]{dk}).
\begin{theorem}\label{thmpower}
Let $E$ be a Banach space with a normalized, shrinking, $1$-unconditional basis $(e_n)_{n=1}^\infty$ such that for some $p\in [1,\infty)$ its dual $E^\ast$ is asymptotic $\ell_{p}$ with respect to $(e_n^\ast)_{n=1}^\infty$. Then for every power type bounded sequence $(X_n)_{n=1}^\infty$ of Banach spaces we have
$$
\mathsf{p}\biggl(\!\Bigl(\bigoplus_{n=1}^\infty X_n\Bigr)_{\!\! E}\biggr)=\max\bigl\{p,\,\mathfrak{p}(X_n)_{n=1}^\infty\bigr\}.
$$
\end{theorem}

As it was shown in \cite[Ex.~5.12 and~5.13]{dk}, the assumption that $E^\ast$ is asymptotic $\ell_p$ with respect to the dual basis is generally essential. However, in the case where all $X_n$'s are finite-dimensional one can reduce the assumptions on $E$ to a~minimum.
\begin{theorem}
Let $E$ be a Banach space with a normalized $1$-unconditional basis $(e_n)_{n=1}^\infty$ and assume $\Sz(E)=\omega$. Then for any sequence $(F_n)_{n=1}^\infty$ of finite-dimensional Banach spaces we have
$$
\mathsf{p}\biggl(\!\Bigl(\bigoplus_{n=1}^\infty F_n\Bigr)_{\!\! E}\biggr)=\mathsf{p}(E).
$$
\end{theorem}

\begin{proof}
Set $X=(\bigoplus_{n=1}^\infty F_n)_E$ and fix any $\e\in (0,1)$ such that $\iota_\e^N B_{X^\ast}\neq\varnothing$ for some $N\in\N$. Then, by Lemma \ref{gkl_lemma}, there is a~weak$^\ast$-null tree $(\mathbf{x}_\alpha^\ast)_{\alpha\in S_N}$ on $X^\ast$ of order $N$ such that $\n{\mathbf{x}_\alpha^\ast}\geq\frac{1}{4}\e$ for each $\alpha\in S_N$ and $\n{\sum_{\alpha\in\Gamma}\mathbf{x}_a^\ast}\leq 1$ for each branch $\Gamma\subset S_N$. Since $\Sz(E)=\omega$, the space $E^\ast$ is separable and hence the basis $(e_n)_{n=1}^\infty$ is shrinking (see, {\it e.g.}, \cite[Thm. 3.3.1]{ak}). Therefore, $X^\ast=(\bigoplus_{n=1}^\infty F_n^\ast)_{E^\ast}$ and we shall write $\mathbf{x}_\alpha^\ast=(x_{\alpha,n}^\ast)_{n=1}^\infty$ for $\alpha\in S_N$.

By slightly shrinking $\e$, if necessary, we may assume that all $\mathbf{x}^\ast_\alpha$'s are finitely supported (with respect to the FDD $(F_n^\ast)_{n=1}^\infty$). Further, by using an easy prunning procedure, we may also assume that the nodes along any branch have disjoint supports ({\it i.e.} each branch yields a~block sequence with respect to $(F_n^\ast)_{n=1}^\infty$). Consequently, we can build a weak$^\ast$-null tree, still denoted by $(\mathbf{x}^\ast_\alpha)_{\alpha\in S_N}$, on $X^\ast$ of order $N$ such that the following conditions are satisfied:
\begin{enumerate}
\item[\rm(i)] $\n{\mathbf{x}^\ast_\alpha}\geq\frac18\e$ for each $\alpha\in S_N$;
\vspace*{2mm}
\item[\rm(ii)] $\big\|\sum_{\alpha\in\Gamma}\mathbf{x}^\ast_\alpha\big\| \leq1$ for each branch $\Gamma\subset S_N$;
\vspace*{2mm}
\item[\rm(iii)] $(\mathbf{x}^\ast_\alpha)_{\alpha\in\Gamma}$ is a block sequence for each branch $\Gamma\subset S_N$.
\end{enumerate}
(Condition (ii) is satisfied, because the basis $(e_n^\ast)_{n=1}^\infty$ is also $1$-unconditional.)

Choose any $q>\mathsf{p}(E)$. By \cite[Thm. 4.8]{gkl} (see also \cite[Prop. 2.3]{dk}), there exist a~norm $\abs{\,\cdot\,}$ on $E$ and $c>0$ with the following properties:
\begin{itemize}
\item[\rm(i)] $\frac12\n{x}\leq\abs{x}\leq\n{x}$ for every $x\in E$;
\vspace*{2mm}
\item[\rm(ii)] if $x^\ast\in E^\ast$ and $(x_n^\ast)_{n=1}^\infty$ is a~weak$^\ast$-null sequence with $\abs{x_n^\ast}\geq\tau$ for some $\tau>0$ and every $n\in\N$, then
$$
\liminf_{n\to\infty}\abs{x^\ast+x_n^\ast}\geq\big(\abs{x^\ast}^q+c\tau^q\big)^{1/q}.
$$
\end{itemize}
(We use the same symbol for the corresponding dual norm; of course, $\n{x^\ast}\leq \abs{x^\ast}\leq 2\n{x^\ast}$ for every $x^\ast\in E^\ast$.)

Since $F_n$'s are finite-dimensional, we have $\lim_{\nu\to\infty}\n{x_{\alpha\smallfrown \nu,n}^\ast}=0$ for any $\alpha\in S_{N-1}$ and $n\in\N$. It means that each sequence of the form $(\sum_{n=1}^\infty \n{x_{\alpha\smallfrown\nu,n}^\ast}e_n^\ast)_{\nu>\max \alpha}$, where $\alpha\in S_{N-1}$, is weak$^\ast$-null. Therefore, using properties of the tree $(\mathbf{x}^\ast_\alpha)_{\alpha\in S_N}$ and the norm $\abs{\,\cdot\,}$, we obtain

\begin{equation*}
\begin{split}
1 &\geq \liminf_{\nu_2\to\infty} \cdots \liminf_{\nu_N\to\infty} \Bigg\|\sum_{j=1}^N \mathbf{x}_{(\nu_1,\ldots,\nu_j)}^\ast\Bigg\| = \liminf_{\nu_2\to\infty} \cdots \liminf_{\nu_N\to\infty} \Bigg\|\sum_{n=1}^\infty\bigg\|\sum_{j=1}^N x_{(\nu_1,\ldots,\nu_j),n}^\ast\bigg\| e_n^\ast\Bigg\| \\
&= \liminf_{\nu_2\to\infty} \cdots \liminf_{\nu_N\to\infty} \Bigg\|\sum_{n=1}^\infty\sum_{j=1}^N \big\|x_{(\nu_1,\ldots,\nu_j),n}^\ast\big\| e_n^\ast\Bigg\|\\
&\geq \frac12\liminf_{\nu_2\to\infty} \cdots \liminf_{\nu_N\to\infty} \Bigg|\sum_{j=1}^N\sum_{n=1}^\infty \big\|x_{(\nu_1,\ldots,\nu_j),n}^\ast\big\| e_n^\ast\Bigg|\\
&\geq\frac12\left(\bigg|\sum_{n=1}^\infty \big\|x_{(\nu_1),n}^\ast\big\| e_n^\ast\bigg|^q+c(N-1)\Big(\frac18\e\Big)^q\right)^{\! 1/q} \geq \frac{1}{16}(cN)^{1/q}\e.
\end{split}
\end{equation*}
Thus, for some constant $C>0$, we have $N\leq C\e^{-q}$ and hence $\mathsf{p}(X)\leq q$.
\end{proof}
We conclude the paper with posing a~question closely related to our present work. 

\begin{problem}
Suppose $X$ and $Y$ are separable Banach spaces with summable Szlenk index. Does necessarily $X\inj Y$ have summable Szlenk index as well?
\end{problem}

\subsection*{Acknowledgements}
The first-named author was supported by the University of Silesia Mathematics Department (Iterative Functional Equations and Real Analysis program).

\bibliographystyle{amsplain}

\end{document}